\documentclass[a4paper,twoside]{article}
\usepackage{a4}
\usepackage{amssymb}
\usepackage{amsmath}
\usepackage{amsthm}
\usepackage{mathtools}
\usepackage{url}
\usepackage{graphicx}
\usepackage{pgf, tikz}
\usepackage{float}
\usetikzlibrary{arrows, automata, matrix, shapes}

\usepackage{authblk}

\newtheorem{theorem}{Theorem}[section]
\newtheorem{lemma}[theorem]{Lemma}

\newtheorem{remark}{Remark}[section]
\newtheorem{example}{Example}[section]

\DeclareMathOperator{\divergence}{div}
\DeclareMathOperator*{\supp}{supp}

\title{Preservation of Piecewise Constancy under TV Regularization with Rectilinear Anisotropy\thanks{The final authenticated publication is available online at \newline\protect\url{https://doi.org/10.1007/978-3-030-22368-7_40}}}
\author[1]{Clemens Kirisits}
\author[1]{Eric Setterqvist}
\author[1,2]{Otmar Scherzer}
\affil[1]{Faculty of Mathematics, University of Vienna, Vienna, Austria}
\affil[2]{Johann Radon Institute for Computational and Applied Mathematics (RICAM), Austrian Academy of Sciences, Linz, Austria}
\date{March 11, 2019}
	
\begin{document}

\maketitle

\begin{abstract}
A recent result by \L{}asica, Moll and Mucha about the $\ell^1$-anisotropic Rudin-Osher-Fatemi model in $\mathbb{R}^2$ asserts that the solution is piecewise constant on a rectilinear grid, if the datum is.
By means of a new proof we extend this result to $\mathbb{R}^n$.
The core of our proof consists in showing that averaging operators associated to certain rectilinear grids map subgradients of the $\ell^1$-anisotropic total variation seminorm to subgradients.
\end{abstract}

\section{Introduction}\label{sec:intro}

This article is concerned with a variant of the Rudin-Osher-Fatemi (ROF) image denoising model \cite{Rudin1}. More specifically, we consider minimization of
\begin{equation}\label{eq:l1rof}
	\frac{1}{2}\|u-f\|^2_{L^2} + \alpha J(u),
\end{equation}
where $J(u) = \int_\Omega \|\nabla u(x)\|_{\ell^1}dx$ is the total variation with $\ell^1$-anisotropy. This model and variations thereof have been used in imaging applications for data exhibiting a rectilinear geometry \cite{BerBurDroNem06,choksi2011,SanOzkRomGok18,Setzer1}. Numerical algorithms for minimizing \eqref{eq:l1rof} have been studied, for example, in \cite{chen2015,GolOsh09,LiSan96}.

The $\ell^1$-anisotropic total variation has a special property from a theoretical point of view as well. It has been shown in \cite[Thm.~3.4, Rem.~3.5]{Casas99} that approximation of a general $u\in BV\cap L^p$ by functions $u_m$ piecewise constant on rectilinear grids, in the sense that
	$$ \| u - u_m\|_{L^p} \to 0 \qquad \text{and} \qquad J(u_m) \to J(u), $$
is not possible for $J(u) = \int_\Omega \|\nabla u(x)\|_{\ell^q}dx$, unless $q=1$.

Let $\Omega\subset \mathbb{R}^n$ be a finite union of hyperrectangles, each aligned with the coordinate axes. Our main result, Theorem \ref{thm:fpcrupcr}, states that if the given function $f:\Omega \to \mathbb{R}$ is piecewise constant on a rectilinear grid, then the minimizer of \eqref{eq:l1rof} is too. This extends a recent result by \L{}asica, Moll and Mucha about two-dimensional domains \cite[Thm. 5]{Lasica1}. Their proof is based on constructing the solution by means of its level sets and relies on minimization of an anisotropic Cheeger-type functional over subsets of $\Omega$. 

The proof we present below is centred around the averaging operator $A_G$ associated to the grid $G$ on which $f$ is piecewise constant. In addition to being a contraction, it has the crucial property of mapping subgradients of $J$ to subgradients, that is, $A_G (\partial J(0)) \subset \partial J(0)$, see Theorem \ref{prop:projection}. Combined with the dual formulation of \eqref{eq:l1rof} we obtain that the minimizer must be piecewise constant on the same grid as $f$. While it might be possible to extend the techniques of \cite{Lasica1} to higher dimensions, we believe that modifying the so-called ``squaring step" in the proof of \cite[Lem.~2]{Lasica1} could lead to difficulties.

Theorem \ref{thm:fpcrupcr} implies that, if $f$ is piecewise constant on a rectilinear grid, then minimization of functional \eqref{eq:l1rof} becomes a finite-dimensional problem. More precisely, in this case the solution can be found by minimizing a discrete energy of the form
\begin{align*}
	\sum_i w_i |u_i - f_i|^2 + \alpha \sum_{i,j} w_{ij} |u_i-u_j|,
\end{align*}
where the weights $w_i,w_{ij}\ge 0$ depend only on the grid. For problems of this sort there are many efficient algorithms, such as graph cuts \cite{Cha05,ChaDar09,DarSig06a}. Extending the preservation of piecewise constancy to domains $\Omega \subset \mathbb{R}^n$, $n\ge 3$, means that the discrete reformulation can also be exploited for processing higher dimensional data such as volumetric images or videos.

This article is organized as follows. Section \ref{sec:setting} contains the basic concepts that will be required throughout. In Section \ref{sec:pcr} we define several spaces of piecewise constant functions, while Section \ref{sec:rof} is devoted to the $\ell^1$-anisotropic ROF model. Section \ref{sec:main} is the main part of this paper. It starts with introducing the averaging operator $A_G$ and ends with Theorem \ref{thm:fpcrupcr}. The article is concluded in Section \ref{sec:conclusion}.

\section{Mathematical preliminaries} \label{sec:setting}

\subsection{PCR functions} \label{sec:pcr}

In this section we introduce several notions related to functions which are piecewise constant on rectilinear subsets of $\mathbb{R}^n$. Some of these are $n$-dimensional analogues of notions from \cite[Sec.\ 2.3]{Lasica1}.

A bounded set $R\subset \mathbb{R}^n$ which can be written as a Cartesian product of $n$ proper intervals is called an \emph{$n$-dimensional hyperrectangle.} Recall that an interval is proper, if it is neither empty nor a singleton. Finite unions of $n$-dimensional hyperrectangles will be referred to as \emph{rectilinear $n$-polytopes}.

A \emph{rectilinear grid}, or simply \emph{grid}, is a finite family of affine hyperplanes, each being perpendicular to one of the coordinate axes of $\mathbb{R}^n$. For a rectilinear $n$-polytope $P$ we denote by $G(P)$ the smallest grid with the property that the union of all its affine hyperplanes contains the entire boundary of $P$.

Throughout this article $\Omega \subset \mathbb{R}^n$ is an \emph{open} rectilinear $n$-polytope. A finite family of rectilinear $n$-polytopes $\mathcal{Q}=\{P_1,\ldots, P_N\}$ is called a \emph{partition} of $\Omega$, if they have pairwise disjoint interiors and the union of their closures equals $\overline{\Omega}$. Every grid $G$ defines a partition $\mathcal{Q}(G)$ of $\Omega$ into rectilinear $n$-polytopes in the following way: $P\subset \Omega$ belongs to $\mathcal{Q}(G)$, if and only if its boundary is contained in $\partial \Omega \cup \bigcup G$ while $\mathrm{int}\, P $ and $ \bigcup G$ are disjoint. Note that if $G$ contains $G(\Omega)$, then $\mathcal{Q}(G)$ consists only of hyperrectangles. 

We adopt the notation $PCR(\Omega)$, or simply $PCR$, from \cite{Lasica1} for the set of all integrable functions $f:\Omega \to \mathbb{R}$ which can be written as finite linear combinations of indicator functions of rectilinear $n$-polytopes. That is, $f \in PCR$ if there is an $N\in \mathbb{N}$, $c_i \in \mathbb{R}$ and rectilinear $n$-polytopes $P_i \subset \Omega$, $1 \le i \le N$, such that
\begin{equation}\label{eq:pcrfunction}
	f = \sum_{i=1}^N c_i \mathbf{1}_{P_i}
\end{equation}
almost everywhere. Here, $\mathbf{1}_{A}$ is the indicator function of the set $A$, defined by
$$	\mathbf{1}_{A}(x) =
	\begin{cases}
	1,	& x\in A, \\
	0,	& x \notin A.
	\end{cases}
$$
We can assume, without loss of generality, that the values $c_i$ are pairwise distinct and that the polytopes $P_i$ are pairwise disjoint, which makes the representation \eqref{eq:pcrfunction} unique almost everywhere. To every $f \in PCR$ we associate its \emph{minimal grid}, that is, the unique smallest grid covering the boundaries of all level sets of $f$, given by
\begin{equation*}
	G_f = \bigcup_{i=1}^N G(P_i).
\end{equation*}
Note that the partition $\mathcal{Q}(G_f)$ always consists of hyperrectangles only. See Figure \ref{fig:grid} for an illustration of $G_f$ and $\mathcal{Q}(G_f).$
\begin{figure}[t]
 \centering
\begin{tikzpicture}[scale=0.5]
 \draw[fill=black!10] (0,0) -- (5.5,0) -- (5.5,3) -- (0,3) -- cycle;
 \draw[fill=black!20] (6,0) -- (6,3) -- (7.5,3) -- (7.5,6) -- (9,6) -- (9,0) -- cycle;
 \draw[fill=black!30] (0,3) -- (0,7) -- (9,7) -- (9,6) -- (7.5,6) -- (7.5,3) -- cycle;
 \draw[fill=black!60] (2,2.5) -- (2,5) -- (4,5) -- (4,2.5) -- cycle;
 \draw[fill=black!40] (5.5,0) -- (5.5,4) -- (6,4) -- (6,0) -- cycle;
\end{tikzpicture}
\qquad
\begin{tikzpicture}[scale=0.5]
 \draw[line width=1pt] (0,0) -- (5.5,0) -- (5.5,3) -- (4,3);
 \draw[line width=1pt] (0,0) -- (0,3);
 \draw[line width=1pt] (6,0) -- (6,3) -- (7.5,3) -- (7.5,6) -- (9,6) -- (9,0) -- cycle;
 \draw[line width=1pt] (2,3)-- (0,3) -- (0,7) -- (9,7) -- (9,6) -- (7.5,6) -- (7.5,3) -- (6,3);
 \draw[line width=1pt] (2,2.5) -- (2,5) -- (4,5) -- (4,2.5) -- cycle;
 \draw[line width=1pt] (5.5,0) -- (5.5,4) -- (6,4) -- (6,0) -- cycle;
 \end{tikzpicture}
 
\vspace{2em}
\begin{tikzpicture}[scale=0.5]
 \draw[line width=1pt] (0,0) -- (5.5,0) -- (5.5,3) -- (4,3);
 \draw[line width=1pt] (0,0) -- (0,3);
 \draw[line width=1pt] (6,0) -- (6,3) -- (7.5,3) -- (7.5,6) -- (9,6) -- (9,0) -- cycle;
 \draw[line width=1pt] (2,3)-- (0,3) -- (0,7) -- (9,7) -- (9,6) -- (7.5,6) -- (7.5,3) -- (6,3);
 \draw[line width=1pt] (2,2.5) -- (2,5) -- (4,5) -- (4,2.5) -- cycle;
 \draw[line width=1pt] (5.5,0) -- (5.5,4) -- (6,4) -- (6,0) -- cycle;
 \draw[dashed,line width=0.25pt] (2,0) -- (2,7);
 \draw[dashed,line width=0.25pt] (4,0) -- (4,7);
 \draw[dashed,line width=0.25pt] (2,0) -- (2,7);
 \draw[dashed,line width=0.25pt] (4,0) -- (4,7);
 \draw[dashed,line width=0.25pt] (0,2.5) -- (9,2.5);
 \draw[dashed,line width=0.25pt] (0,5) -- (9,5);
 \draw[dashed,line width=0.25pt] (0,2.5) -- (9,2.5);
 \draw[dashed,line width=0.25pt] (0,5) -- (9,5);
 \draw[dashed,line width=0.25pt] (6,0) -- (6,7);
 \draw[dashed,line width=0.25pt] (0,6) -- (9,6);
 \draw[dashed,line width=0.25pt] (7.5,0) -- (7.5,7);
 \draw[dashed,line width=0.25pt] (0,3) -- (9,3);
 \draw[dashed,line width=0.25pt] (7.5,0) -- (7.5,7);
 \draw[dashed,line width=0.25pt] (0,3) -- (9,3);
 \draw[dashed,line width=0.25pt] (0,3) -- (9,3);
 \draw[dashed,line width=0.25pt] (0,4) -- (9,4);
 \draw[dashed,line width=0.25pt] (0,4) -- (9,4);
 \draw[dashed,line width=0.25pt] (5.5,0) -- (5.5,7);
\end{tikzpicture}
\qquad
\begin{tikzpicture}[scale=0.5]
 \draw[line width=1pt] (0,0) -- (5.5,0) -- (5.5,3) -- (4,3);
 \draw[line width=1pt] (0,0) -- (0,3);
 \draw[line width=1pt] (6,0) -- (6,3) -- (7.5,3) -- (7.5,6) -- (9,6) -- (9,0) -- cycle;
 \draw[line width=1pt] (2,3)-- (0,3) -- (0,7) -- (9,7) -- (9,6) -- (7.5,6) -- (7.5,3) -- (6,3);
 \draw[line width=1pt] (2,2.5) -- (2,5) -- (4,5) -- (4,2.5) -- cycle;
 \draw[line width=1pt] (5.5,0) -- (5.5,4) -- (6,4) -- (6,0) -- cycle;
 \draw[line width=1pt] (2,0) -- (2,2.5);
 \draw[line width=1pt] (4,0) -- (4,2.5);
 \draw[line width=1pt] (2,5) -- (2,7);
 \draw[line width=1pt] (4,5) -- (4,7);
 \draw[line width=1pt] (0,2.5) -- (2,2.5);
 \draw[line width=1pt] (0,5) -- (2,5);
 \draw[line width=1pt] (4,2.5) -- (9,2.5);
 \draw[line width=1pt] (4,5) -- (9,5);
 \draw[line width=1pt] (6,3) -- (6,7);
 \draw[line width=1pt] (0,6) -- (7.5,6);
 \draw[line width=1pt] (7.5,6) -- (7.5,7);
 \draw[line width=1pt] (2,3) -- (4,3);
 \draw[line width=1pt] (7.5,3) -- (7.5,0);
 \draw[line width=1pt] (7.5,3) -- (9,3);
 \draw[line width=1pt] (5.5,3) -- (6,3);
 \draw[line width=1pt] (0,4) -- (5.5,4);
 \draw[line width=1pt] (6,4) -- (9,4);
 \draw[line width=1pt] (5.5,4) -- (5.5,7);
\end{tikzpicture}
 \caption{Upper left: A PCR function $f$ on a planar domain $\Omega$. Each level set of $f$ is visualized using a different grey tone. Upper right: The boundaries of the level sets of $f$.
 Lower left: Extension of the boundaries of the level sets of $f$. Lower right: The minimal grid $G_{f}$ (lines) and the partition $\mathcal{Q}(G_{f})$ of $\Omega$ (rectangular cells).}
 \label{fig:grid}
\end{figure}
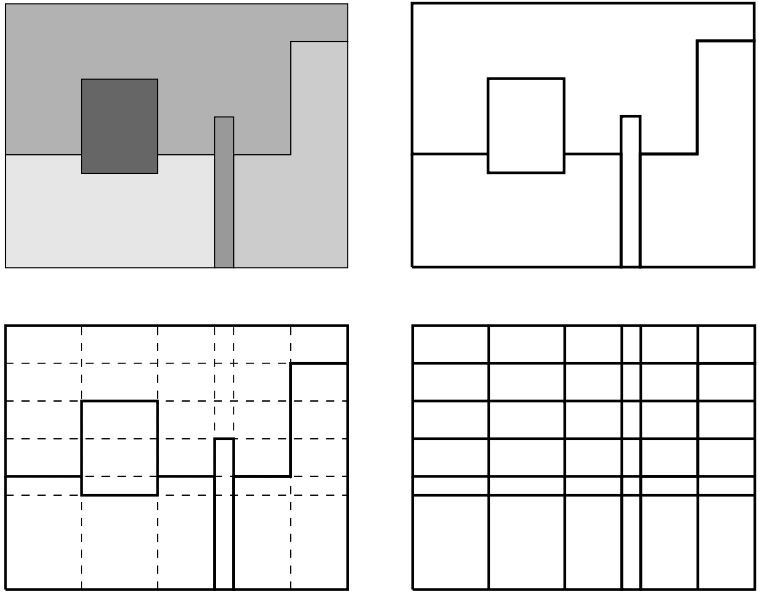

For a given grid $G$ we denote by $PCR_G$ the set of all functions in $PCR$ which are equal almost everywhere to a finite linear combination of indicator functions of $P_i \in \mathcal{Q}(G)$.

The following notions are essential for the proof of Theorem \ref{prop:projection}.
Fix an $i \in \{1,\ldots,n\}$ as well as coordinates $x_1,\ldots,x_{i-1},x_{i+1},\ldots,x_n$.
The set $\{x_i \in \mathbb{R} : (x_1,\ldots,x_n) \in \Omega\}$ is a union of finitely many disjoint intervals
	\begin{equation}\label{eq:intervals}
			I_i^k = (a_i^k,b_i^k), \quad k=1,\ldots,m.
	\end{equation}
Note that the number of intervals $m$ as well as the intervals $I_i^k$ themselves depend on, and are uniquely determined by, the coordinates $x_1,\ldots,x_{i-1},x_{i+1},\ldots,x_n$. For an illustration of the intervals $I_i^k$ see Example \ref{example} below.
Next, let $\mathcal{G}$ be the set of all rectilinear grids of $\mathbb{R}^n$. For every $G\in \mathcal{G}$ and $i \in \{1,\ldots,n\}$ we define
\begin{align*}
	\Gamma^{i}_{G} &= \left\{ g\in PCR_{G} : \underset{s\in (a^k_{i},b^k_{i})}{\sup}\left|\int^{s}_{a^k_{i}}g\,dx_{i}\right|\leq 1, \int^{b^k_{i}}_{a^k_{i}}g\,dx_{i}=0, 1\le k \le m \right\}.
\end{align*}
The restrictions on $g$ are to be understood for almost every
$$(x_1,\ldots,x_{i-1},x_{i+1},\ldots,x_n) \in \mathbb{R}^{n-1} $$
such that there is an $x_i\in \mathbb{R}$ satisfying $(x_1,\ldots,x_n) \in \Omega.$ The sum of the spaces $\Gamma^{i}_{G}$ is denoted by
\begin{align*} 
	\Gamma_{G} = \left\{\sum^{n}_{i=1}g_{i} : g_{i}\in\Gamma^{i}_{G}, 1 \le i \le n \right\},
\end{align*}
and we further set
$$\Gamma=\bigcup_{G\in \mathcal{G}}\Gamma_{G}.$$
\begin{remark}
The set $\Gamma_{G}$ consists of divergences of certain piecewise affine vector fields. More precisely, Theorem \ref{prop:projection} below implies that $\Gamma_G = \partial J(0) \cap PCR_G$, that is, $\Gamma_G$ is the set of all subgradients of $J$ which are piecewise constant on $G$.
\end{remark}
Finally, those elements of $\Gamma^{i}_{G}$ which have compact support in $\Omega$ are collected in the set $\Gamma^{i}_{G,c}$, and we define analogously
\begin{align*}
	\Gamma_{G,c}	&=	\left\{\sum^{n}_{i=1}g_{i} : g_{i}\in\Gamma^{i}_{G,c}, 1 \le i \le n \right \}, \\
	\Gamma_c		&=	\bigcup_{G\in \mathcal{G}}\Gamma_{G,c}.
\end{align*}
\begin{example}\label{example}
For the rectilinear $2$-polytope $\Omega$ of Figure \ref{fig:intervals} we have the following intervals $I_1^k$ and $I_2^k$
\begin{align*}
\{x_1 \in \mathbb{R} : (x_1,x_2) \in \Omega\}=
	\begin{cases}
		\left(0,6\right), & \text{if } x_2\in\left(0,1\right)\cup\left(2,3\right),\\
		\left(0,2\right)\cup\left(4,6\right), & \text{if } x_2\in\left[1,2\right],\\
		\left(3,6\right), & \text{if } x_2\in\left[3,4\right),
	\end{cases}
\end{align*}
and
\begin{align*}
\{x_2 \in \mathbb{R} : (x_1,x_2) \in \Omega\}=
	\begin{cases}
		\left(0,3\right), & \text{if } x_1\in\left(0,2\right),\\
		\left(0,1\right)\cup\left(2,3\right), & \text{if } x_1\in\left[2,3\right],\\
		\left(0,1\right)\cup\left(2,4\right), & \text{if } x_{1}\in\left(3,4\right],\\
		\left(0,4\right), & \text{if } x_1\in\left(4,6\right).
	\end{cases}
\end{align*}
\begin{figure}
\centering
\begin{tikzpicture}
 \draw[fill=black!10] (0,0) -- (6,0) -- (6,4) -- (3,4) -- (3,3) -- (0,3) -- (0,0);
 \draw[fill=white] (2,1) -- (4,1) -- (4,2) -- (2,2) -- (2,1);
 \draw	(0,0) node{$\bullet$};
 \draw	(0,-0.25) node{$(0,0)$};
 \draw	(6,0) node{$\bullet$};
 \draw	(6,-0.25) node{$(6,0)$};
 \draw	(6,4) node{$\bullet$};
 \draw	(6,4.25) node{$(6,4)$};
 \draw	(3,4) node{$\bullet$};
 \draw	(3,4.25) node{$(3,4)$};
 \draw	(3,3) node{$\bullet$};
 \draw	(3,2.75) node{$(3,3)$};
 \draw	(0,3) node{$\bullet$};
 \draw	(0,3.25) node{$(0,3)$};
 \draw	(2,1) node{$\bullet$};
 \draw	(2,0.75) node{$(2,1)$};
 \draw	(4,1) node{$\bullet$};
 \draw	(4,0.75) node{$(4,1)$};
 \draw	(4,2) node{$\bullet$};
 \draw	(4,2.25) node{$(4,2)$};
 \draw	(2,2) node{$\bullet$};
 \draw	(2,2.25) node{$(2,2)$};
 \end{tikzpicture}
 \caption{A rectilinear 2-polytope $\Omega$.}
 \label{fig:intervals}
\end{figure}
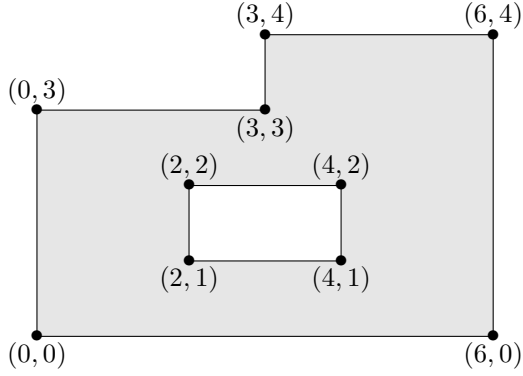
\end{example}

\subsection{The $\ell^1$-anisotropic ROF model} \label{sec:rof}

The notion of anisotropic total variation was introduced in \cite{AmaBel94}. In this article we exclusively consider one particular variant.

For every $\alpha > 0$ we denote by $\mathcal{B}_{\alpha}$ the set of all smooth compactly supported vector fields on $\Omega$ whose components are bounded by $\alpha$, that is,
\begin{align*} 
\mathcal{B}_{\alpha}=\left\{H\in\ C^{\infty}_{c}\left(\Omega,\mathbb{R}^{n}\right):\underset{1\leq i\leq n}{\max}\left|H_{i}(x)\right|\leq\alpha,\forall x\in\Omega
\right\}.
\end{align*}
The $\ell^1$-\emph{anisotropic total variation} $J:L^2(\Omega) \to \mathbb{R}\cup\{+\infty\}$ is given by
\begin{equation} \label{eq:tvsupportfunction}
	J(u) = \sup_{H\in \mathcal{B}_1} \int_\Omega u \divergence H \, dx = \sup_{h \in \overline{ \divergence \mathcal{B}_1}} \int_\Omega u h \, dx,
\end{equation}
where the bar denotes closure in $L^2(\Omega)$. Thus, $J$ is the support function of the closed and convex set $\overline{\divergence \mathcal{B}_1}$, which implies that $\overline{\divergence \mathcal{B}_1} = \partial J (0)$, or more generally
\begin{equation}\label{eq:divBa}
	\overline{\divergence \mathcal{B}_\alpha} = \alpha \partial J (0)
\end{equation}
for every $\alpha>0.$ If $u$ is a Sobolev function, then $J(u) = \int_\Omega \|\nabla u(x)\|_{\ell^1}dx.$

The next lemma states that the $\ell^1$-anisotropic ROF model is equivalent to constrained $L^2$-minimization. However, the way it is formulated it actually applies to every support function $J$ of a closed and convex subset of $L^2(\Omega)$.
\begin{lemma} \label{lemma:rofdual}
For every $\alpha>0$ and $f\in L^2{(\Omega)}$ the minimization problem
\begin{equation}\label{eq:rof}
	\min_{u\in L^2(\Omega)} \frac{1}{2}\|u-f\|^2_{L^2} + \alpha J(u)
\end{equation}
is equivalent to
	$$ \min_{u\in f - \alpha \partial J(0)} \|u\|_{L^2} .$$
\end{lemma}
\begin{proof}
The dual problem associated to \eqref{eq:rof} is given by
\begin{equation}\label{eq:rofdual}
	\min_{w\in L^2(\Omega)} \frac{1}{2}\|w-f\|^2_{L^2} + (\alpha J)^*(w),
\end{equation}
where the asterisk stands for convex conjugation. The two solutions $u_\alpha$ and $w_\alpha$ of \eqref{eq:rof} and \eqref{eq:rofdual}, respectively, satisfy the optimality conditions
\begin{equation}\label{eq:rofoptimality}
\begin{aligned}
	u_\alpha	&= f - w_\alpha, \\
	w_\alpha	&\in \partial (\alpha J)(u_\alpha).
\end{aligned}
\end{equation}
Concerning the derivation of \eqref{eq:rofdual} and \eqref{eq:rofoptimality} we refer to \cite[Chap.\ III, Rem.\ 4.2]{Ekeland2}. Since $\alpha J$ is the support function of the set $\alpha \partial J(0)$, recall \eqref{eq:tvsupportfunction}, its conjugate is the characteristic function
\begin{equation*}
	(\alpha J)^*(w) =
	\begin{cases}
		0, 		& w \in \alpha \partial J(0), \\
		+\infty,& w \notin \alpha \partial J(0).
	\end{cases}
\end{equation*}
Therefore, problem \eqref{eq:rofdual} is equivalent to
\begin{equation*}
	\min_{w\in \alpha \partial J(0)} \|w-f\|_{L^2}.
\end{equation*}
Finally, using the optimality condition \eqref{eq:rofoptimality} we get
$$ \|u_\alpha\|_{L^2} = \|w_\alpha-f\|_{L^2} = \min_{w\in \alpha \partial J(0)}\|w-f\|_{L^2} = \min_{u\in f -  \alpha \partial J(0)}\|u\|_{L^2}.$$
\end{proof}

\section{The averaging operator $A_G$}\label{sec:main}

Let $\Omega$ be a rectilinear $n$-polytope and $G$ a grid. Define the averaging operator $A_{G}:L^{1}(\Omega)\rightarrow PCR_{G}\left(\Omega\right)$ by
\begin{align*}
 A_{G}g=\sum^{N}_{i=1}\left(\frac{1}{\left|P_{i}\right|}\int_{P_{i}}g(s)ds\right)\mathbf{1}_{P_{i}},
\end{align*}
where $P_i \in \mathcal{Q}(G)$ and $\left|P_{i}\right|$ is its $n$-dimensional volume. 
\par Two properties of the operator $A_{G}$ turn out to be important when establishing the main result of this paper, Theorem \ref{thm:fpcrupcr}. The first one is
\begin{lemma} \label{lemma:operatorA}
For every $u\in L^{1}(\Omega)$ and convex $\varphi:\mathbb{R} \to \mathbb{R}$
\begin{align*}
 \int_{\Omega}\varphi\left((A_{G}u)(x)\right)dx\leq\int_{\Omega}\varphi\left(u(x)\right)dx.
 \end{align*}
\end{lemma}
\begin{proof}
By applying Jensen's inequality we obtain
\begin{align*}
 \int_{\Omega}\varphi\left((A_{G}u)(x)\right)dx &=\sum^{N}_{i=1}\int_{P_{i}}\varphi\left((A_{G}u)(x)\right)dx=\sum^{N}_{i=1}\varphi\left(\frac{1}{\left|P_{i}\right|}\int_{P_{i}}u(x)dx\right)\left|P_{i}\right|\\
 &\leq\sum^{N}_{i=1}\int_{P_{i}}\varphi(u(x))dx=\int_{\Omega}\varphi(u(x))dx.
\end{align*}
\end{proof}
\begin{remark}\label{rem:operatorA}
Lemma \ref{lemma:operatorA} implies in particular that $A_G$ is a contraction,
\begin{align*}
 \lVert A_{G}\rVert_{L^{p}\rightarrow L^{p}}\leq 1, \quad 1\leq p < \infty.
\end{align*}
\end{remark}
The second property of $A_G$ is that it maps subgradients of $J$ to subgradients. Recall that $G(\Omega)$ is the smallest grid covering the entire boundary of $\Omega$.
\begin{theorem} \label{prop:projection}
Let $G$ be a grid containing $G(\Omega)$. Then $A_{G}\left(\partial J(0)\right) \subset \partial J(0)$.
\end{theorem}
\begin{proof}
The proof is divided into three steps, each being proved in a separate lemma
	$$ 	A_{G}\left(\partial J(0)\right) \stackrel{\mathrm{Lem.\ }\ref{lemma:step1}}{\subset}
		\Gamma_G \subset
		\Gamma \stackrel{\mathrm{Lem.\ }\ref{lemma:step2}}{\subset}
		\overline{\Gamma_c} \stackrel{\mathrm{Lem.\ }\ref{lemma:step3}}{\subset}
		\partial J(0).
	$$
Note that the inclusion $\Gamma_G \subset \Gamma$ is trivial.
\end{proof}

\begin{lemma} \label{lemma:step1}
Let $G$ be a grid containing $G(\Omega)$. Then $A_{G}\left(\partial J (0)\right)\subset\Gamma_{G}$.
\end{lemma}
\begin{proof}
Throughout this proof we exploit the fact that $\partial J (0) = \overline{\divergence \mathcal{B}_1}$, recall equation \eqref{eq:divBa}.

First, note that $PCR_G$ is a finite-dimensional subspace of $L^2$ and that the sets $\Gamma^{i}_{G}$, $i=1,...,n$, are bounded and closed subsets of $PCR_{G}$. It follows that $\Gamma_{G}$ is a closed subset of $PCR_G$ and in particular of $L^2$. Therefore, it suffices to show $A_{G}\left(\divergence\mathcal{B}_{1}\right)\subset\Gamma_{G}$, as we then have
$ A_G ( \overline{\divergence\mathcal{B}_{1}}) \subset \overline{A_G (\divergence\mathcal{B}_{1} )} \subset \overline{\Gamma_{G}} = \Gamma_G $, because $A_G$ is continuous.

Take $H=(H_1,\ldots,H_n)\in \mathcal{B}_{1}$. We want to show that $A_G\partial H_i /\partial x_i \in \Gamma_G^i$, that is,
\begin{align*} 
\underset{s\in\left(a^k_{i},b^k_{i}\right)}{\sup}\left|\int^{s}_{a^k_{i}}A_{G} \frac{\partial H_{i}}{\partial x_{i}} dx_{i}\right| \leq 1, \qquad \text{and} \qquad
 \int^{b^k_{i}}_{a^k_{i}}A_{G}\frac{\partial H_{i}}{\partial x_{i}}dx_{i}=0,
\end{align*}
for $i=1,\ldots,n$ and each $k$, where $\left(a^k_{i},b^k_{i}\right)$ are the intervals defined in equation \eqref{eq:intervals}.

Consider the second integral first. From the definition of $A_G$ it follows that $(a^k_i,b^k_i)$ can be divided into a finite number of subintervals in such a way that the integrand is constant on each. In addition the assumption $G \supset G(\Omega)$ implies that the partition $\mathcal{Q}(G)$ consists of hyperrectangles only. Thus, after a potential relabelling of the $R_j\in\mathcal{Q}(G)$, we can write
\begin{align*} 
	\int^{b^k_{i}}_{a^k_{i}}A_{G}\frac{\partial H_{i}}{\partial x_{i}} dx_{i}
		&= \sum_{j=1}^M \int_{s_{j-1}}^{s_j} \left( \frac{1}{|R_j|} \int_{R_j} \frac{\partial H_{i}}{\partial x_{i}} \,dx \right) \, dx_i \nonumber \\
		&= \sum_{j=1}^M \frac{s_j-s_{j-1}}{|R_j|} \int_{R_j} \frac{\partial H_{i}}{\partial x_{i}} \,dx
	\intertext{for some $ M \in \{1,\ldots, N\}$ and $a^k_i = s_0 < s_1 < \cdots < s_M = b^k_i.$ Note that $|R_j|/(s_j-s_{j-1})$ is the $(n-1)$-dimensional volume of $\partial R_j \cap \partial R_{j+1}$, and that this volume is independent of $j \in \{1,\ldots M\}$. In other words, the hyperrectangles $R_j$ only differ in their extent in $x_i$-direction, compare Figure \ref{fig:grid}, bottom right. The reason is that $\mathcal{Q}(G)$ is not an arbitrary partition of $\Omega$ into hyperrectangles, but rather formed by a grid. Setting $C=(s_j-s_{j-1})/|R_j|$ we further obtain} \nonumber
		&= C \int_{\bigcup_j R_j} \frac{\partial H_{i}}{\partial x_{i}} \,dx. \nonumber
	\intertext{The remaining integral can be computed by turning it into an iterated one, integrating with respect to $x_i$ first, and recalling that $H$ is compactly supported}
		&= C \underbracket[.5pt]{\idotsint}_{n-1} \int_{a_i^k}^{b_i^k} \frac{\partial H_{i}}{\partial x_{i}} \, dx_i
	 		= C \idotsint H_{i}\Big|_{x_i=a_i^k}^{x_i=b_i^k} = 0. \nonumber
\end{align*}
Here $F \big|_{x_i=a}^{x_i=b}$ stands for $F(x_i=b)-F(x_i=a)$, where $F(x_i=c)$ denotes the restriction of $F$ to the affine hyperplane defined by $x_i=c$.

Now integrate up to an arbitrary $s\in(a^k_i,b^k_i]$. We can assume $s\in(s_{\ell-1},s_\ell]$ for some $\ell\in\{1,\ldots,M\}$ and a brief computation similar to the one above shows that
\begin{align*}
	\int^s_{a^k_{i}}A_{G}\frac{\partial H_{i}}{\partial x_{i}} dx_{i}
		&= C \idotsint H_{i}\Big|_{x_i=a_i^k}^{x_i=s_{\ell-1}} + \frac{s-s_{\ell-1}}{|R_\ell|} \idotsint H_{i}\Big|_{x_i=s_{\ell-1}}^{x_i=s} \\
		&= C \idotsint H_{i}(x_i=s_{\ell-1}) + \frac{s-s_{\ell-1}}{|R_\ell|} \idotsint H_{i}\Big|_{x_i=s_{\ell-1}}^{x_i=s}.
	\intertext{Recalling that we can write $C = (s_\ell - s_{\ell-1})/|R_\ell|$ we rearrange terms}
		&= \frac{s_\ell-s}{|R_\ell|} \idotsint H_{i}(x_i=s_{\ell-1}) + \frac{s-s_{\ell-1}}{|R_\ell|} \idotsint H_{i}(x_i=s).
	\intertext{Finally, we estimate $H_i \le 1$ and obtain}
		&\le  \frac{s_\ell-s}{|R_\ell|} \frac{|R_\ell|}{s_\ell-s_{\ell-1}} + \frac{s-s_{\ell-1}}{|R_\ell|} \frac{|R_\ell|}{s_\ell-s_{\ell-1}} = 1.
\end{align*}
Similarly, we get $\int^s_{a^k_{i}}A_{G}\frac{\partial H_{i}}{\partial x_{i}} dx_{i} \ge -1$. Thus we have $A_{G}\partial H_{i} / \partial x_{i} \in \Gamma_G^i$.
\end{proof}

\begin{lemma}\label{lemma:step2}
$\Gamma\subset\overline{\Gamma_{c}}$.
\end{lemma}
\begin{proof}
Let $j\in \mathbb{N}$ and define $\Omega_j \subset \Omega$ by removing strips of width $1/j$ from the boundary of $\Omega$. It is assumed that $j$ is chosen large enough such that the strips are contained in $\Omega$. See Figure \ref{fig:strip1} for an example of the construction of $\Omega_j$ in the plane.
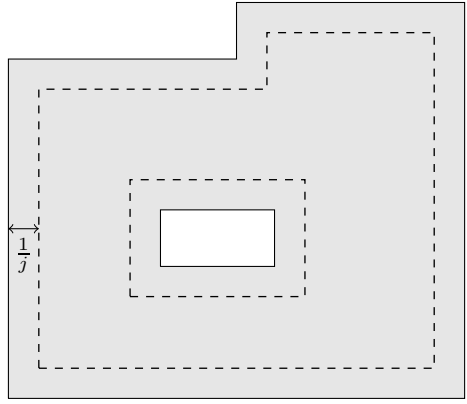
\begin{figure}
\centering
\begin{tikzpicture}
 \draw[fill=black!10] (-2,-1.25) -- (4,-1.25) -- (4,4) -- (1,4) -- (1,3.25) -- (-2,3.25) -- (-2,-1.25);
 \draw[fill=white] (0,0.5) -- (1.5,0.5) -- (1.5,1.25) -- (0,1.25) -- (0,0.5);
 \draw[dashed,line width=0.5pt] (-1.6,-0.85) -- (3.6,-0.85) -- (3.6,3.6) -- (1.4,3.6) -- (1.4,2.85) -- (-1.6,2.85) -- (-1.6,-0.85);
 \draw[dashed,line width=0.5pt] (-0.4,0.1) -- (1.9,0.1) -- (1.9,1.65) -- (-0.4,1.65) -- (-0.4,0.1);
\draw[<->] (-2,1) -- node[below] {$\frac{1}{j}$} (-1.6,1);
 \end{tikzpicture}
 \caption{The construction of $\Omega_j$ by removing strips of width $1/j$ from $\Omega$.}
 \label{fig:strip1}
\end{figure}
Next, for $i=1,\ldots,n$ we define $\Omega^i_j\subset \Omega_j$, by removing strips of width $1/j$ from those parts of the boundary of $\Omega_j$ which are orthogonal to the $x_{i}$-axis. By choosing $j$ large enough, the strips will be contained in $\Omega_j$. For an illustration of the construction of $\Omega^i_j$ in the plane, see Figure \ref{fig:strip2}.
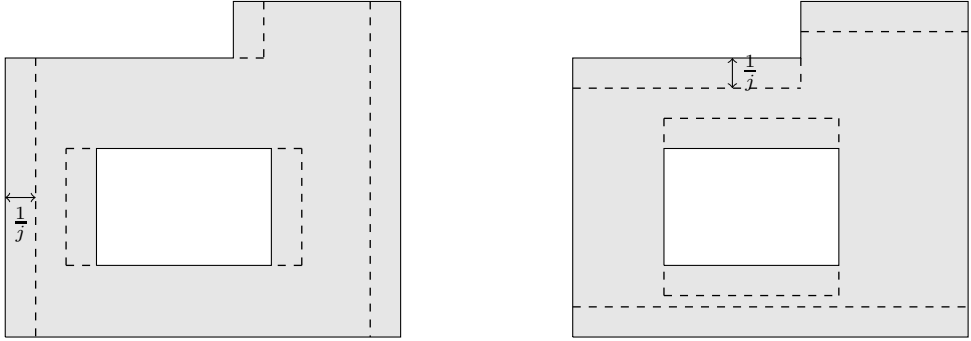
\begin{figure}
\centering
\begin{tikzpicture}
 \draw[fill=black!10] (-1.6,-0.85) -- (3.6,-0.85) -- (3.6,3.6) -- (1.4,3.6) -- (1.4,2.85) -- (-1.6,2.85) -- (-1.6,-0.85);
 \draw[fill=white] (-0.4,0.1) -- (1.9,0.1) -- (1.9,1.65) -- (-0.4,1.65) -- (-0.4,0.1);
 \draw[dashed,line width=0.5pt] (3.2,3.6) -- (3.2,-0.85);
 \draw[dashed,line width=0.5pt] (1.8,3.6) -- (1.8,2.85);
 \draw[dashed,line width=0.5pt] (1.8,2.85) -- (1.4,2.85);
 \draw[dashed,line width=0.5pt] (-1.2,2.85) -- (-1.2,-0.85);
 \draw[dashed,line width=0.5pt] (-0.8,0.1) -- (-0.4,0.1);
 \draw[dashed,line width=0.5pt] (2.3,0.1) -- (1.9,0.1);
 \draw[dashed,line width=0.5pt] (2.3,0.1) -- (2.3,1.65);
 \draw[dashed,line width=0.5pt] (2.3,1.65) -- (1.9,1.65);
 \draw[dashed,line width=0.5pt] (-0.8,1.65) -- (-0.4,1.65);
 \draw[dashed,line width=0.5pt] (-0.8,0.1) -- (-0.8,1.65);
 \draw[<->] (-1.6,1) -- node[below] {$\frac{1}{j}$} (-1.2,1);
 \end{tikzpicture}
 \hfill
 \begin{tikzpicture}
 \draw[fill=black!10] (-1.6,-0.85) -- (3.6,-0.85) -- (3.6,3.6) -- (1.4,3.6) -- (1.4,2.85) -- (-1.6,2.85) -- (-1.6,-0.85);
 \draw[fill=white] (-0.4,0.1) -- (1.9,0.1) -- (1.9,1.65) -- (-0.4,1.65) -- (-0.4,0.1);
 \draw[dashed,line width=0.5pt] (-1.6,-0.45) -- (3.6,-0.45);
 \draw[dashed,line width=0.5pt] (1.4,2.85) -- (1.4,2.45);
 \draw[dashed,line width=0.5pt] (-1.6,2.45) -- (1.4,2.45);
 \draw[dashed,line width=0.5pt] (1.4,3.2) -- (3.6,3.2);
 \draw[dashed,line width=0.5pt] (-0.4,-0.3) -- (1.9,-0.3);
 \draw[dashed,line width=0.5pt] (1.9,-0.3) -- (1.9,0.1);
 \draw[dashed,line width=0.5pt] (1.9,2.05) -- (1.9,1.65);
 \draw[dashed,line width=0.5pt] (-0.4,2.05) -- (1.9,2.05);
 \draw[dashed,line width=0.5pt] (-0.4,2.05) -- (-0.4,1.65);
 \draw[dashed,line width=0.5pt] (-0.4,-0.3) -- (-0.4,0.1);
 \draw[<->] (0.5,2.85) -- node[right] {$\frac{1}{j}$} (0.5,2.45);
 \end{tikzpicture}
 \caption{The construction of $\Omega^1_{j}$ (left) and $\Omega^2_{j}$ (right) by removing strips from $\Omega_j$.}
 \label{fig:strip2}
\end{figure}

Take $h\in\Gamma$. So, $h\in\Gamma_{G}$ for some $G\in\mathcal{G}$ and in particular $h=\sum^{n}_{i=1}h_{i}$ where $h_{i}\in\Gamma^{i}_{G}$. Let $g_{j}=\sum^{n}_{i=1}g_{j,i}$ where
\begin{align*}
	g_{j,i}(x) =
	\begin{cases}
		0,			& \text{if } x\in\Omega\setminus\Omega_j,\\
		2h_{i}(x),	& \text{if } x\in \Omega_j \setminus \Omega^i_j,\\
		h_{i}(x),	& \text{otherwise.}
	\end{cases}
\end{align*}
Note that there is a grid $G_{j} \supset G$ such that $g_{j,i}\in PCR_{G_j}$ for every $i=1,\ldots,n$, and that for $j$ large enough
\begin{align*}
	\left| \int_{a_i^k}^{s} g_{j,i} \, dx_i \right| \le \left| \int_{a_i^k}^{s} h_i \, dx_i \right|
\end{align*}
for every interval $(a_i^k,b_i^k)$, recall equation \eqref{eq:intervals}, and $s\in (a_i^k,b_i^k]$.
It follows that $g_{j,i}\in\Gamma^{i}_{G_{j},c}$ and therefore $g_{j}\in\Gamma_{G_{j},c}$. Finally, it can be directly verified that
\begin{align*}
	\underset{j\rightarrow\infty}{\lim }\lVert g_{j}-h\rVert_{L^{2}}=0.
\end{align*}
As $h\in\Gamma$ was chosen arbitrarily we conclude that $\Gamma\subset\overline{\Gamma_{c}}$.
\end{proof}

\begin{lemma}\label{lemma:step3}
$\overline{\Gamma_{c}} \subset \partial J (0)$.
\end{lemma}
\begin{proof}
As in Lemma \ref{lemma:step1} we use the fact that $\partial J (0) = \overline{\divergence\mathcal{B}_{1}}.$

Take $h\in\Gamma_{c}$. So there is a grid $G$ such that $h=\sum^{n}_{i=1}h_{i}\in\Gamma_{G,c}$ where $h_{i}\in\Gamma^{i}_{G,c}$.
From $h$ we now construct a vector field $H=\left(H_{1},\ldots,H_{n}\right)$.
For every $i \in \{1,\ldots,n\}$ and $x \in \Omega$ there is a unique interval $I_i^k = (a_i^k,b_i^k)$ containing $x_i,$ recall equation \eqref{eq:intervals}. Based on this observation we define the components of $H$ by
\begin{align*}
	H_{i}(x)=\int^{x_{i}}_{a^k_i}h_{i}(x_1,\ldots,x_{i-1},s,x_{i+1}\ldots, x_n)\,ds.
\end{align*}
It follows that $\lVert H_{i}\rVert_{L^{\infty}}\leq1$ and $\supp(H_{i})\subset\Omega$. $H$ is now modified into a vector field belonging to $\mathcal{B}_{1}$. Let $\{\rho_{j}\}_{j\in\mathbb{N}}$ denote a sequence of mollifiers on $\mathbb{R}^{n}$ supported on the closed Euclidean ball centred at $0$ with radius $1/j$. Recalling standard results regarding convolution and mollifiers, we derive  $\lVert H_{i}*\rho_{j}\rVert_{L^{\infty}}\leq\lVert H_{i}\rVert_{L^{\infty}}\lVert\rho_{j}\rVert_{L^{1}}=\lVert H_{i}\rVert_{L^{\infty}}\leq1$ and moreover, for $j$ large enough, $H_{i}*\rho_{j}\in C^{\infty}_{c}\left(\Omega\right)$. Hence, for $j\in\mathbb{N}$ large enough, the modification $H_{\rho_{j}}$ of $H$ given by
\begin{align*}
	H_{\rho_{j}}=\left(H_{1}*\rho_{j},\ldots,H_{n}*\rho_{j}\right)
\end{align*}
is in $\mathcal{B}_{1}$. It follows that $h\in\overline{\divergence\mathcal{B}_{1}}$, as
\begin{align*}
	\left\| \divergence H_{\rho_{j}}-h \right\|_{L^{2}}
		&=	\Big\| \sum_{i=1}^n \Big( \frac{\partial}{\partial x_{i}} (H_{i}*\rho_{j}) - h_i \Big) \Big\|_{L^{2}}
			= \Big\| \sum_{i=1}^n \left( h_{i}*\rho_{j} - h_i \right) \Big\|_{L^{2}} \\
		&\le\sum_{i=1}^n \Big\| h_{i}*\rho_{j} - h_i \Big\|_{L^{2}} \xrightarrow{j\to\infty} 0.
\end{align*}
The element $h\in\Gamma_{c}$ was chosen arbitrarily and $\overline{\divergence\mathcal{B}_{1}}$ is closed, therefore 
$\overline{\Gamma_{c}}\subset\overline{\divergence\mathcal{B}_{1}}$.
\end{proof}

\subsection{Preservation of piecewise constancy}
We are now ready to prove the following result.
\begin{theorem}\label{thm:fpcrupcr}
Given $f\in PCR$ with minimal grid $G_{f}$, the minimizer $u_{\alpha}$ of the corresponding anisotropic ROF functional
\begin{equation*}
	\min_{u\in L^2(\Omega)} \frac{1}{2}\|u-f\|^2_{L^2} + \alpha J(u)
\end{equation*}
lies in $PCR_{G_{f}}$.
\end{theorem}
\begin{proof}
Recall that, according to Lemma \ref{lemma:rofdual}, $u_{\alpha}$ is the unique element with minimal $L^{2}$-norm in $f - \alpha\partial J(0)$. From Theorem \ref{prop:projection} and the fact that $A_{G_{f}}f = f$ it follows that also $A_{G_{f}}u_{\alpha}\in f - \alpha\partial J(0)$. As $\lVert A_{G_f}u_{\alpha}\rVert_{L^{2}}\leq\lVert u_{\alpha}\rVert_{L^{2}}$, because of Remark \ref{rem:operatorA}, we have $A_{G_{f}}u_{\alpha}=u_{\alpha}$. Therefore, $u_{\alpha}\in PCR_{G_{f}}$.
\end{proof}

\section{Conclusion} \label{sec:conclusion}

In \cite[Thm.\ 5]{Lasica1} the authors have shown that, for $\Omega$ being a rectilinear 2-polytope, $f \in PCR$ implies $u_\alpha \in PCR$. We have extended this preservation of piecewise constancy to rectilinear $n$-polytopes. Our proof can be summarized in the following way
	$$	\| A_{G_f} u_\alpha \|_{L^2}
			\stackrel{\text{Lem.~\ref{lemma:operatorA}}}{\le}
		\| u_\alpha \|_{L^2}
			\stackrel{\text{Lem.~\ref{lemma:rofdual}}}{=}
		\min_{u \in f - \alpha \partial J(0)} \| u \|_{L^2}
			\stackrel{\text{Thm.~\ref{prop:projection}}}{\le}
		\| A_{G_f} u_\alpha \|_{L^2}.
	$$
The crucial step is Theorem \ref{prop:projection}, asserting that
\begin{equation*} 
	A_G (\partial J(0)) \subset \partial J(0),
\end{equation*}
which exploits the fact that the anisotropy of $J$ is compatible with the rectilinearity of the grid $G$.

\subsection*{Acknowledgements}
We acknowledge support by the Austrian Science Fund (FWF) within the national research network ``Geometry $+$ Simulation," S117, subproject 4.

\end{document}